\documentclass[10pt]{amsart}

\usepackage{amsmath,amsthm,amsfonts,latexsym,amssymb,mathrsfs,setspace,enumerate}

\newtheorem{thm}{Theorem}

\newtheorem{lem}[thm]{Lemma}

\numberwithin{thm}{section}
\numberwithin{equation}{section}

\theoremstyle{definition}

\newcommand{\rat}{\mathbb Q}

\newcommand{\com}{\mathbb C}
\newcommand{\alg}{\overline\rat}
\newcommand{\algt}{\alg^{\times}}

\newcommand{\nat}{\mathbb N}

\newcommand{\G}{\mathcal G}
\newcommand{\tor}{\mathrm{Tor}}

\newcommand{\gal}{\mathrm{Gal}}
\newcommand{\norm}{\mathrm{Norm}}
\newcommand{\rad}{\mathrm{Rad}}
\newcommand{\comment}[1]{}

\title{The infimum in the metric Mahler measure}
\author[C.L. Samuels]{Charles L. Samuels}
\address{Max-Planck-Institut f\"ur Mathematik, Vivatsgasse 7, 53111 Bonn, Germany}
\email{csamuels@mpim-bonn.mpg.de}

\begin{document}

\begin{abstract}
	Dubickas and Smyth defined the metric Mahler measure on the multiplicative group of non-zero algebraic numbers.
	The definition involves taking an infimum over representations of an algebraic number $\alpha$ by other
	algebraic numbers.  We verify their conjecture that the infimum in its definition is always achieved as well as establish its
	analog for the ultrametric Mahler measure.
\end{abstract}

\maketitle

\section{Introduction}

Let $K$ be a number field and $v$ a place of $K$ dividing the place $p$ of $\rat$.  Let $K_v$ and $\rat_p$
denote the respective completions.  We write $\|\cdot\|_v$ for the unique absolute value on $K_v$ 
extending the $p$-adic absolute value on $\rat_p$ and define 
\begin{equation*}
  |\alpha|_v=\|\alpha \|_v^{[K_v:\rat_p]/[K:\rat]}
\end{equation*}
for all $\alpha\in K$.  Define the {\it Weil height} of $\alpha\in K$ by
\begin{equation*} \label{WeilHeightDef}
  H(\alpha) = \prod_v\max\{1,|\alpha|_v\}
\end{equation*}
where the product is taken over all places $v$ of $K$.  Given this normalization of our absolute values,
the above definition does not depend on $K$, and therefore, $H$ is a well-defined function on $\alg$.
Clearly $H(\alpha)\geq 1$, and by Kronecker's Theorem, we have equality precisely when $\alpha$ is zero 
or a root of unity.  It is obvious that if $\zeta$ is a root of unity then
\begin{equation} \label{HeightRoot}
	H(\alpha) = H(\zeta\alpha),
\end{equation}
and further, if $n$ is an integer then it is well-known that
\begin{equation} \label{HeightInteger}
	H(\alpha^n) = H(\alpha)^{|n|}.
\end{equation}
Also, if $\alpha,\beta\in\algt$ then $H(\alpha\beta) \leq H(\alpha)H(\beta)$ so that $H$ 
satisfies the multiplicative triangle inequality.

We further define the {\it Mahler measure} of an algebraic number $\alpha$ by $$M(\alpha) = H(\alpha)^{[\rat(\alpha):\rat]}.$$
Since $H$ is invariant under Galois conjugation over $\rat$, we obtain immediately
\begin{equation} \label{MahlerProduct}
	M(\alpha) = \prod_{n=1}^N H(\alpha_n),
\end{equation}
where $\alpha_1,\ldots,\alpha_N$ are the conjugates of $\alpha$ over $\rat$.  Further, it is well-known that
\begin{equation} \label{SecondMahlerDefinition}
	M(\alpha) = |A|\cdot\prod_{n=1}^N\max\{1,|\alpha_n|\},
\end{equation}
where $|\cdot |$ denotes the usual absolute value on $\com$.  While the right hand side of \eqref{SecondMahlerDefinition} appears initially to depend upon a particular 
embedding of $\alg$ into $\com$, any change of embedding simply permutes the images of the points $\{\alpha_n\}$ so that \eqref{SecondMahlerDefinition} 
remains unchanged.

It follows, again from Kronecker's Theorem, that $M(\alpha) = 1$ if and only if $\alpha$ is zero or a root of unity.  As part of an algorithm for
computing large primes, D.H. Lehmer \cite{Lehmer} asked whether there exists a constant $c>1$ such that $M(\alpha) \geq c$ in all other cases.
The smallest known Mahler measure greater than $1$ occurs at a root of 
\begin{equation*}
  \ell(x) = x^{10}+x^9-x^7-x^6-x^5-x^4-x^3+x+1
\end{equation*}
which has Mahler measure $1.17\ldots$.  
Although an affirmative answer to Lehmer's problem has been given in many special cases, the general case remains open.  The best known universal lower 
bound on $M(\alpha)$ is due to Dobrowolski \cite{Dobrowolski}, who proved that 
\begin{equation} \label{Dobrowolski}
  \log M(\alpha) \gg \left(\frac{\log\log \deg \alpha}{\log \deg \alpha}\right)^3
\end{equation}
whenever $\alpha$ is not a root of unity.

Recently, Dubickas and Smyth \cite{DubSmyth} defined the {\it metric Mahler measure} of an algebraic number $\alpha$ by
\begin{equation} \label{MetricDef}
	M_1(\alpha) = \inf\left\{\prod_{n=1}^NM(\alpha_n):N\in\nat,\ \alpha_n\in\algt,\ \alpha = \prod_{n=1}^N\alpha_n\right\}.
\end{equation}
Here, the infimum is taken over all ways to represent $\alpha$ as a product of elements in $\algt$.
It is easily verified that $$M_1(\alpha\beta) \leq M_1(\alpha)M_1(\beta)$$ for all $\alpha,\beta\in\algt$, and further,
$M_1$ is well-defined on the quotient group $\G = \algt/\tor(\algt)$.  This implies that the map
$(\alpha,\beta)\mapsto\log M_1(\alpha\beta^{-1})$ defines a metric on $\G$ which induces the discrete topology if and only if
there is an affirmative answer to Lehmer's problem.  

Also in \cite{DubSmyth}, Dubickas and Smyth conjecture that the infimum in the definition of $M_1$ is always achieved.  We verify this
conjecture as well as explicitly determine a set in which the infimum must occur.

If $K$ is any number field let
\begin{equation*}
	\rad(K) = \left\{\alpha\in\algt:\alpha^r\in K\mathrm{\ for\ some}\ r\in\nat\right\},
\end{equation*}
the set of all roots of points in $K$.  Also, for the remainder of this paper,
we write $K_\alpha$ for the Galois closure of $\rat(\alpha)$ over $\rat$.

\begin{thm} \label{Achieved}
	If $\alpha$ is a non-zero algebraic number then there exist $\alpha_1,\ldots,\alpha_N\in\rad(K_\alpha)$ such that 
	$\alpha = \alpha_1\cdots\alpha_N$ and $M_1(\alpha) = M(\alpha_1)\cdots M(\alpha_N)$.
\end{thm}

Motivated by the work of Dubickas and Smyth, Fili and the author \cite{FiliSamuels} defined a non-Archimedean version of $M_1$ by
replacing the product in \eqref{MetricDef} by a maximum.  That is, define the {\it ultrametric Mahler measure} by
\begin{equation*}
	M_\infty(\alpha) = \inf\left\{\max_{1\leq n\leq N}M(\alpha_n):N\in\nat,\ \alpha_n\in\algt,\ \alpha = \prod_{n=1}^N\alpha_n\right\}.
\end{equation*}
It easily verified that $M_\infty$ satisfies the strong triangle inequality $$M_\infty(\alpha\beta)\leq\max\{M_\infty(\alpha),M_\infty(\beta)\}$$ for all
non-zero algebraic numbers $\alpha$ and $\beta$.  It is further shown in \cite{FiliSamuels} that $M_\infty$ is well-defined on the quotient
group $\G$.  We can now establish the obvious analog of Theorem \ref{Achieved} for $M_\infty$.

\begin{thm} \label{AchievedToo}
	If $\alpha$ is a non-zero algebraic number then there exist $\alpha_1,\ldots,\alpha_N\in\rad(K_\alpha)$ such that 
	$\alpha = \alpha_1\cdots\alpha_N$ and $M_\infty(\alpha) = \max\{M(\alpha_1),\ldots,M(\alpha_N)\}$.
\end{thm}

The remainder of this paper is organized in the following way.  Section \ref{SimpleRep} contains the core of our argument
in which we show that computing $M_1(\alpha)$ and $M_\infty(\alpha)$ requires only the use of elements in $\rad(K_\alpha)$.
In section \ref{Proofs}, we finish the proofs of Theorems \ref{Achieved} and \ref{AchievedToo} by showing, essentially, that
there are only finitely many values for the Mahler measure in $\rad(K_\alpha)$.  Finally, section \ref{Applications} contains 
some applications of these results, giving the location of the algebraic numbers $M_1(\alpha)$ and $M_\infty(\alpha)$.

\section{Reducing to simpler representations} \label{SimpleRep}

The main idea in both proofs involves a method for replacing an arbitrary representation of $\alpha$ by a potentially smaller representation
containing only points in $\rad(K_\alpha)$.  This technique is summarized by the following result.

\begin{thm} \label{Reduction}
	If $\alpha,\alpha_1,\ldots,\alpha_N$ are non-zero algebraic numbers with $\alpha = \alpha_1\cdots\alpha_N$ then
	there exists a root of unity $\zeta$ and algebraic numbers $\beta_1,\ldots,\beta_N$ satifying
	\begin{enumerate}[(i)]
		\item $\alpha = \zeta\beta_1\cdots\beta_N$,
		\item $\beta_n\in \rad(K_\alpha)$ for all $n$,
		\item $M(\beta_n) \leq M(\alpha_n)$ for all $n$.
	\end{enumerate}
\end{thm}

It is worth noting that we are unaware of an example in which computing $M_1(\alpha)$ or $M_\infty(\alpha)$ requires the use of elements outside
$K_\alpha$.  Hence, it seems reasonable to believe that we can, in fact, choose the points $\beta_n$ to belong to $K_\alpha$.  
Unfortunately, our proof suggests no way to verify this.

The proof of Theorem \ref{Reduction} is based on the following lemma.

\begin{lem} \label{Contained}
	Suppose that $K$ is Galois over $\rat$.  If $\gamma$ is an algebraic number then
	\begin{equation}\label{EqualDegrees}
		[K(\gamma):K] = [\rat(\gamma):K\cap\rat(\gamma)].
	\end{equation}
	Moreover, we have that 
	\begin{equation}\label{ContainedEq}
		\prod_{n=1}^N\gamma_n \in K \cap \rat(\gamma),
	\end{equation}
	where $\gamma_1,\ldots,\gamma_N$ are the conjugates of $\gamma$ over $K$.
\end{lem}
\begin{proof}
	We see clearly that $K(\gamma)$ is the compositum of $K$ and $\rat(\gamma)$.  Since $K$ is Galois over $\rat$, it follows 
	(see \cite{DF}, p. 505, Prop. 19) that $[K(\gamma):K] = [\rat(\gamma):K\cap\rat(\gamma)]$, verifying \eqref{EqualDegrees}.
	We also observe that
	\begin{equation*}
		(K\cap\rat(\gamma))(\gamma) \subseteq (\rat(\gamma))(\gamma) = \rat(\gamma) \subseteq (K\cap\rat(\gamma))(\gamma)
	\end{equation*}
	so we conclude from \eqref{EqualDegrees} that
	\begin{equation} \label{EqualDegreesAgain}
		[K(\gamma):K] = [(K\cap\rat(\gamma))(\gamma):K\cap\rat(\gamma)].
	\end{equation}
	
	Let $f$ be the monic minimal polynomial of $\gamma$ over $K\cap\rat(\gamma)$ so that $f$ has
	degree $D$ equal to both sides of \eqref{EqualDegreesAgain}.  Now write
	\begin{equation*}
		f(x) = x^D +\cdots +a_1x + a_0
	\end{equation*}
	and note that $f$ is, of course, a polynomial over $K$.  In fact, $f$ is the monic minimal polynomial of $\gamma$ over $K$ 
	because it vanishes at $\gamma$ and has degree $[K(\gamma):K]$.  Since $\gamma_1,\ldots,\gamma_N$ are the conjugates 
	of $\gamma$ over $K$ we conclude that
	\begin{equation*}
		\prod_{n=1}^N\gamma_n = \pm a_0
	\end{equation*}
	which belongs to $K\cap\rat(\gamma)$.
\end{proof}

It is worth observing that if $\rat(\gamma)$ is Galois over $\rat$, then Lemma \ref{Contained} becomes trivial.  Indeed, $\gamma_1\cdots\gamma_N$ certainly
belongs to $K$ by definition.  But also, if $\rat(\gamma)$ is Galois, then $\rat(\gamma)$ contains all conjugates of $\gamma$ over $\rat$.  In particular,
it contains $\gamma_n$ for all $n$, so it contains their product as well.  Of course, the proof of Theorem \ref{Reduction} does not permit such a hypothesis, so 
we require the above lemma.

Additionally, we cannot omit the hypothesis that $K$ be Galois over $\rat$.  For example, let $\gamma_1, \gamma_2$ and $\gamma_3$ be the roots of a third degree, 
irreducible polynomial over $\rat$ having Galois group $S_3$.  This means that $\rat(\gamma_1) \cap \rat(\gamma_2) = \rat$.
Further, we observe that $\gamma_2$ must have degree $2$ over $\rat(\gamma_1)$ implying that its conjugates over this field are $\gamma_2$ and $\gamma_3$.  
But if $\gamma_2\cdot\gamma_3\in\rat(\gamma_2)$ then $\gamma_1\in\rat(\gamma_2)$, a contradiction.

\smallskip\smallskip

\noindent{\it Proof of Theorem \ref{Reduction}}. Suppose that $\alpha = \alpha_1\cdots\alpha_N$ and let $E$ be a Galois extension
of $K_\alpha$ containing $\alpha_n$ for all $n$.  Let $G = \gal(E/K_\alpha)$, $G_n = \gal(E/K_\alpha(\alpha_n))$ and $S_n$ a set of left
coset representatives of $G_n$ in $G$.  We have that
\begin{align*}
	\alpha ^{[E:K_\alpha]} & = \norm_{E/K_\alpha}(\alpha) \\
		& = \prod_{n=1}^N \norm_{E/K_\alpha}(\alpha_n)\\
		& = \prod_{n=1}^N \prod_{\sigma\in G}\sigma(\alpha_n) \\
		& = \prod_{n=1}^N \prod_{\sigma\in S_n}\prod_{\tau\in G_n}\sigma(\tau(\alpha_n)) \\
		& = \prod_{n=1}^N \prod_{\sigma\in S_n}\sigma(\alpha_n)^{|G_n|}
\end{align*}
so we conclude that
\begin{equation} \label{FirstNewFac}
	\alpha^{[E:K_\alpha]} = \prod_{n=1}^N \left(\prod_{\sigma\in S_n}\sigma(\alpha_n)\right)^{[E:K_\alpha(\alpha_n)]}.
\end{equation}
For each $n$, we select an element $\beta_n\in\alg$ such that 
\begin{equation} \label{BetaDef}
	\beta_n^{[K_\alpha(\alpha_n):K_\alpha]} = \prod_{\sigma\in S_n}\sigma(\alpha_n)
\end{equation}
so that, in view of \eqref{FirstNewFac}, we obtain
\begin{equation} \label{NewBetaFac}
	\alpha^{[E:K_\alpha]} = \prod_{n=1}^N \beta_n^{[E:K_\alpha]}.
\end{equation}
This implies the existence of a root of unity $\zeta$ such that
\begin{equation*}
	\alpha = \zeta\beta_1\cdots\beta_N.
\end{equation*}
Furthermore, the set $\{\sigma(\alpha_n): \sigma\in S_n\}$ is precisely the set of conjugates of $\alpha_n$ over $K_\alpha$ so that
\begin{equation*}
	\prod_{\sigma\in S_n}\sigma(\alpha_n) \in K_\alpha.
\end{equation*}
It then follows from \eqref{BetaDef} that $\beta_n\in\rad(K_\alpha)$ for each $n$ as well.  

It remains to show that $M(\beta_n)\leq M(\alpha_n)$ for all $n$.  To see this, we note that \eqref{BetaDef} yields immediately
\begin{equation} \label{DegreeBound1}
	\deg(\beta_n) \leq [K_\alpha(\alpha_n):K_\alpha]\cdot\deg\left(\prod_{\sigma\in S_n}\sigma(\alpha_n)\right).
\end{equation}
Once again, the elements $\sigma(\alpha_n)$ for $\sigma\in S_n$ are precisely the conjugates of $\alpha_n$ over $K_\alpha$.
Hence, we may apply Lemma \ref{Contained} \eqref{EqualDegrees} to find that
\begin{equation*} \label{ContainedAgain}
	\prod_{\sigma\in S_n}\sigma(\alpha_n) \in K_\alpha \cap \rat(\alpha_n).
\end{equation*}
Combining this with \eqref{DegreeBound1}, we obtain
\begin{equation} \label{DegreeBound2}
	\deg(\beta_n) \leq [K_\alpha(\alpha_n):K_\alpha]\cdot[K_\alpha\cap \rat(\alpha_n):\rat].
\end{equation}
Then we find that
\begin{align*}
	M(\beta_n) & \leq H(\beta_n)^{[K_\alpha(\alpha_n):K_\alpha]\cdot[K_\alpha\cap \rat(\alpha_n):\rat]} \\
		& = H\left(\prod_{\sigma\in S_n}\sigma(\alpha_n)\right)^{[K_\alpha\cap \rat(\alpha_n):\rat]} \\
		& \leq H(\alpha_n)^{[K_\alpha(\alpha_n):K_\alpha]\cdot[K_\alpha\cap \rat(\alpha_n):\rat]}
	\end{align*}
where the last inequality follows since the Weil height is invariant under Galois conjugation and
satisfies the triangle inequality.  But also $K_\alpha(\alpha_n)$ is the compositum of $K_\alpha$
and $\rat(\alpha_n)$ so that $[K_\alpha(\alpha_n):K_\alpha] = [\rat(\alpha_n):K_\alpha\cap \rat(\alpha_n)]$
by \eqref{EqualDegrees} in Lemma \ref{Contained}.  This yields
\begin{equation*}
	M(\beta_n) \leq H(\alpha_n)^{[\rat(\alpha_n):K_\alpha\cap \rat(\alpha_n)]\cdot [K_\alpha\cap \rat(\alpha_n):\rat]} = M(\alpha_n)
\end{equation*}
which completes the proof.\qed

\section{Proofs of Theorems \ref{Achieved} and \ref{AchievedToo}} \label{Proofs}

In view of Theorem \ref{Reduction}, it is enough, in the definitions of $M_1$ and $M_\infty$, to consider only representations 
$\alpha = \alpha_1\cdots\alpha_N$ having $\alpha_n\in\rad(K_\alpha)$ for all $n$.  Any representation that fails to have 
this property may simply be replaced a smaller represenation that does.  The remainder of our proofs of both Theorem \ref{Achieved}
and \ref{AchievedToo} require us to show that such representations yield only finitely many different values for
\begin{equation*}
	\max_{1\leq n\leq N}M(\alpha_n)\quad\mathrm{and}\quad \prod_{n=1}^NM(\alpha_n).
\end{equation*}
The following lemma provides the starting point for this argument.

\begin{lem} \label{HeightInK}
	Let $K$ be a Galois extension of $\rat$.  If $\gamma\in\rad(K)$ then there exists a root of unity $\zeta$ and $L,S\in\nat$
	such that $\zeta\gamma^L\in K$ and
	\begin{equation*}
		M(\gamma) = M(\zeta\gamma^L)^S.
	\end{equation*}
	In particular, the set
	\begin{equation*}
		\{M(\gamma):\gamma\in\rad(K),\ M(\gamma) \leq B\}
	\end{equation*}
	is finite for every $B \geq 1$.
\end{lem}
\begin{proof}
	Suppose that $\gamma^r\in K$, so that each conjugate of $\gamma$ over $K$ must be a root of $x^r-\gamma^r\in K[x]$.  
	Therefore, we may assume that $\gamma$ has conjugates
	\begin{equation*}
		\zeta_1\gamma,\ldots,\zeta_L\gamma
	\end{equation*}
	over $K$ for some roots of unity $\zeta_1,\ldots,\zeta_L$.  By Lemma \ref{Contained} we conclude that
	\begin{equation} \label{InK}
		\zeta_1\cdots\zeta_L\gamma^L = \zeta_1\gamma\cdots\zeta_L\gamma \in K\cap \rat(\gamma).
	\end{equation}
	Since $K$ is Galois, Lemma \ref{Contained} also implies that $L = [K(\gamma):K] = [\rat(\gamma):K\cap\rat(\gamma)]$.  Hence, we find that
	\begin{align*}
		M(\gamma) & = H(\gamma)^{[\rat(\gamma):\rat]} \\
			& = H(\gamma)^{[\rat(\gamma):K\cap\rat(\gamma)]\cdot[K\cap\rat(\gamma):\rat]} \\
			& = H(\gamma)^{L\cdot[K\cap\rat(\gamma):\rat]}.
	\end{align*}
	Since $L$ is a positive integer and $\zeta_1\cdots\zeta_L$ is a root of unity, we conclude from \eqref{HeightRoot} and \eqref{HeightInteger} that
	\begin{equation} \label{AlmostMM}
		M(\gamma) = H(\zeta_1\cdots\zeta_L\gamma^L)^{[K\cap\rat(\gamma):\rat]}.
	\end{equation}
	By \eqref{InK} we know that there exists a positive integer $S$ such that
	\begin{equation*}
		[K\cap\rat(\gamma):\rat] = S\cdot [\rat(\zeta_1\cdots\zeta_L\gamma^L):\rat]
	\end{equation*}
	and so \eqref{AlmostMM} yields
	\begin{equation*}
		M(\gamma) = H(\zeta_1\cdots\zeta_L\gamma^L)^{S\cdot [\rat(\zeta_1\cdots\zeta_L\gamma^L):\rat]} = M(\zeta_1\cdots\zeta_L\gamma^L)^S.
	\end{equation*}
	Taking $\zeta = \zeta_1\cdots\zeta_L$ we have that $\zeta\gamma^L\in K$ by \eqref{InK} and $M(\gamma) = M(\zeta\gamma^L)^S$ which 
	establishes the first statement of the lemma.
	
	Further, we note that \eqref{AlmostMM} implies that
	\begin{equation*}
		M(\gamma) = H((\zeta\gamma^L)^{[K\cap\rat(\gamma):\rat]}).
	\end{equation*}
	But $(\zeta\gamma^L)^{[K\cap\rat(\gamma):\rat]} \in K$ implying that
	\begin{equation} \label{SetContain}
		\{M(\gamma):\gamma\in\rad(K),\ M(\gamma) \leq B\} \subseteq \{H(\alpha):\alpha\in K^\times,\ H(\alpha) \leq B\}.
	\end{equation}
	It follows from Northcott's Theorem \cite{Northcott} that the right hand side of \eqref{SetContain} is finite, completing the proof.
\end{proof}

The proof of Theorem \ref{AchievedToo} is somewhat simpler than that of Theorem \ref{Achieved} so we include it here first.

\smallskip\smallskip

\noindent{\it Proof of Theorem \ref{AchievedToo}.}
	There exists $\varepsilon >0$ such that if $\alpha = \alpha_1\cdots\alpha_N$ with $\alpha_n\in\rad(K_\alpha)$ and
	\begin{equation*}
		M_\infty(\alpha) \leq \max\{M(\alpha_1),\ldots,M(\alpha_N)\} \leq M_\infty(\alpha) +\varepsilon
	\end{equation*}
	then $M_\infty(\alpha) = \max\{M(\alpha_1),\ldots,M(\alpha_N)\}$.  Otherwise, we get a sequence $\{x_m\}\subseteq\rad(K_\alpha)$
	such that $\{M(x_m)\}$ is strictly decreasing, contradicting Lemma \ref{HeightInK}.

	By definition, there exists a representation $\alpha = \gamma_1\cdots\gamma_N$ with 
	\begin{equation*}
		M_\infty(\alpha) \leq \max\{M(\gamma_1),\ldots,M(\gamma_N)\} \leq M_\infty(\alpha) +\varepsilon.
	\end{equation*}
	By Theorem \ref{Reduction}, there exists a representation $\alpha = \zeta\alpha_1\cdots\alpha_N$ such that $\zeta$ is a root of unity, $\alpha_n\in\rad(K_\alpha)$ and
	and $M(\alpha_n) \leq M(\gamma_n)$ for all $n$.  This yields
	\begin{equation*}
		M_\infty(\alpha) \leq \max\{M(\alpha_1),\ldots,M(\alpha_N)\} \leq M_\infty(\alpha) +\varepsilon
	\end{equation*}
	so that $M_\infty(\alpha) = \max\{M(\alpha_1),\ldots,M(\alpha_N)\}$ by our earlier remarks.\qed
	
\smallskip\smallskip

We note that the above proof is not sufficient to establish Theorem \ref{Achieved}.  Indeed, Lemma \ref{HeightInK} does not prevent the product 
$M(\alpha_1)\cdots M(\alpha_N)$ from having infinitely many values between $M_1(\alpha)$ and $M_1(\alpha) + \varepsilon$ unless we can bound $N$ uniformly 
from above by a function of $\alpha$.  

In order to do this, we introduce an additional definition.  We say that a representation $\alpha = \alpha_1\cdots\alpha_N$ is 
{\it $B$-restricted} if the following three conditions hold.
\begin{enumerate}[(i)]
	\item $M(\alpha_1)\cdots M(\alpha_N) \leq B$
	\item $\alpha_n\in\rad(K_\alpha)$ for all $n$
	\item At most one element $\alpha_n$ is a root of unity.
\end{enumerate}
We write $R_B(\alpha)$ to denote the set of all $N$-tuples, for all $N\in\nat$, of non-zero algebraic numbers 
that form $B$-restricted representations of $\alpha$.  Further, set 
\begin{equation*}
	q(\alpha) = \inf\left\{H(x):x\in K_\alpha^\times\setminus\tor(\algt)\right\}
\end{equation*}
and note that, by Northcott's Theorem \cite{Northcott}, this quantity is always strictly greater than $1$.
Using these definitions, we obtain the result we need to finish the proof of Theorem \ref{Achieved}.

\begin{lem} \label{FinitelyMany}
	Let $\alpha$ be a non-zero algebraic number and $B\geq 1$.  If $\alpha = \alpha_1\cdots\alpha_N$ is an $B$-restricted representation of $\alpha$ then
	\begin{equation*}
		N \leq 1 + \frac{\log B}{\log q(\alpha)}.
	\end{equation*}
	Moreover, the set
	\begin{equation*}
		\left\{\prod_{n=1}^N M(\alpha_n): N\in\nat,\ (\alpha_1,\ldots,\alpha_N)\in R_B(\alpha)\right\}
	\end{equation*}
	is finite.
\end{lem}
\begin{proof}
	Suppose that $\alpha = \alpha_1\cdots\alpha_N$ is an $B$-restricted representation.  By assumption, at least $N-1$ of the terms $\alpha_n$ in our representation are 
	not roots of unity.  Assume $\alpha_n$ is one such element.  Lemma \ref{HeightInK} implies that there exists a point $\gamma_n\in K_\alpha$ such that 
	\begin{equation} \label{FromK}
		M(\alpha_n) = H(\gamma_n).
	\end{equation}
	Since $\alpha_n$ is not a root of unity, neither side of \eqref{FromK} equals $1$, so that $\gamma_n$ is not a root of unity either.  Therefore,
	we find that $M(\alpha_n) \geq q(\alpha)$ for $N-1$ of the terms belonging to $\{\alpha_1,\ldots,\alpha_N\}$.  This yields
	\begin{equation*}
		B \geq M(\alpha_1)\cdots M(\alpha_N) \geq q(\alpha)^{N-1}.
	\end{equation*}
	We know that $q(\alpha) > 1$ so that we may divide by $\log q(\alpha)$ to obtain
	\begin{equation} \label{NumberBound}
		N \leq 1 + \frac{\log B}{\log q(\alpha)}
	\end{equation}
	verifying the first statement of the lemma.  We now find that
	\begin{align*}
		\left\{ \prod_{n=1}^N M(\alpha_n):\right. &\  N\in\nat,\ (\alpha_1,\ldots,\alpha_N)\in R_B(\alpha) \Bigg\} \\
			& = \left\{ \prod_{n=1}^N M(\alpha_n):\ (\alpha_1,\ldots,\alpha_N)\in R_B(\alpha),\ N \leq 1 + \frac{\log B}{\log q(\alpha)}\right\} \\
			& \subseteq \left\{ \prod_{n=1}^N M(\alpha_n):\ N \leq 1 + \frac{\log B}{\log q(\alpha)},\ M(\alpha_n)\leq B,\ \alpha_n\in\rad(K_\alpha)\right\}
	\end{align*}
	which is finite	by Lemma \ref{HeightInK}.
\end{proof}

{\it Proof of Theorem \ref{Achieved}.}
By Lemma \ref{FinitelyMany}, we may select $B > M_1(\alpha)$ such that
\begin{equation*}
	(M_1(\alpha),B)\ \bigcap\ \left\{\prod_{n=1}^N M(\alpha_n): N\in\nat,\ (\alpha_1,\ldots,\alpha_N)\in R_{M_1(\alpha) +1}(\alpha)\right\} = \emptyset.
\end{equation*}
Of course, we may choose $B\leq M_1(\alpha) + 1$ which gives
\begin{align*}
	\left\{\prod_{n=1}^N M(\alpha_n):\right. & N\in\nat,\ (\alpha_1,\ldots,\alpha_N)\in R_B(\alpha)\Bigg\} \\
	& \subseteq \left\{\prod_{n=1}^NM(\alpha_n): N\in\nat,\ (\alpha_1,\ldots,\alpha_N)\in R_{M_1(\alpha) +1}(\alpha)\right\},
\end{align*}
and therefore,
\begin{equation} \label{EmptyIntersection}
	(M_1(\alpha),B)\ \bigcap\ \left\{\prod_{n=1}^N M(\alpha_n): N\in\nat,\ (\alpha_1,\ldots,\alpha_N)\in R_B(\alpha)\right\} = \emptyset.
\end{equation}
By definition of $M_1$, there exists a representation $\alpha = \gamma_1\cdots\gamma_L$ such that
\begin{equation*}
	M_1(\alpha) \leq M(\gamma_1)\cdots M(\gamma_L) < B.
\end{equation*}
Theorem \ref{Reduction} implies that there exists a representation $\alpha = \zeta\beta_1\cdots\beta_L$ with $\zeta$ a root of unity,
each element $\beta_\ell$ belonging to $\rad(K_\alpha)$ and $M(\beta_\ell) \leq M(\gamma_\ell)$ for all $\ell$.  This yields
\begin{equation*}
	M_1(\alpha) \leq M(\zeta)M(\beta_1)\cdots M(\beta_L) < B.
\end{equation*}
By combining all roots of unity in the representation into a single element, we obtain a new representation $\alpha = \alpha_1\cdots\alpha_N$ having
$\alpha_n\in\rad (K_\alpha)$, at most one root of unity, and
\begin{equation*}
	M(\alpha_1)\cdots M(\alpha_N) = M(\beta_1)\cdots M(\beta_L).
\end{equation*}
Therefore, we see that
\begin{equation} \label{NiceRep}
	M_1(\alpha) \leq M(\alpha_1)\cdots M(\alpha_N) < B
\end{equation}
which implies, in particular, that $(\alpha_1,\ldots,\alpha_N) \in R_B(\alpha)$.  Then by \eqref{EmptyIntersection} we get that
\begin{equation} \label{PushingOut}
	M(\alpha_1)\cdots M(\alpha_N) \not\in (M_1(\alpha),B).
\end{equation}
Finally, combining \eqref{NiceRep} and \eqref{PushingOut} we obtain $M_1(\alpha) = M(\alpha_1)\cdots M(\alpha_N)$.
\qed

\section{The location of $M_1(\alpha)$ and $M_\infty(\alpha)$} \label{Applications}

We now apply Theorems \ref{Achieved} and \ref{AchievedToo} in order to show that $M_1(\alpha)$ and $M_\infty(\alpha)$ belong to $K_\alpha$.
We begin with $M_\infty$ in which case we are able to prove a slightly stronger result.

\begin{thm} \label{StrongLocation}
	If $\alpha$ is an algebraic number then there exists $\beta\in K_\alpha$ such that $M_\infty(\alpha) = M(\beta)$.
	In particular, $M_\infty(\alpha) \in K_\alpha$.
\end{thm}
\begin{proof}
	By Theorem \ref{AchievedToo} there exist $\alpha_1,\ldots,\alpha_N\in\rad(K_\alpha)$ such that $\alpha = \alpha_1\cdots\alpha_N$ and
	\begin{equation*}
		M_\infty(\alpha) = \max\{M(\alpha_1),\ldots,M(\alpha_N)\}.
	\end{equation*}
	For each $n$, Lemma \ref{HeightInK} implies that there exists a root of unity $\zeta_n$ and $L_n,S_n\in\nat$ such that
	\begin{equation*}
		M(\alpha_n) = M(\zeta_n\alpha_n^{L_n})^{S_n}
	\end{equation*}
	and $\zeta_n\alpha_n^{L_n}\in K_\alpha$.  For simplicity, we write
	\begin{equation*}
		L = \prod_{n=1}^N L_n\quad\mathrm{and}\quad J_n = \prod_{k\ne n} L_k
	\end{equation*}
	so that $L = L_nJ_n$ for all $n$.  Then we obtain immediately
	\begin{equation*}
		\alpha^L = \prod_{n=1}^N \alpha_n^{L_nJ_n}
	\end{equation*}
	so there exists a root of unity $\zeta$ such that
	\begin{equation*}
		\zeta\alpha^L = \prod_{n=1}^N (\zeta_n\alpha_n^{L_n})^{J_n}.
	\end{equation*}
	By Theorem 1.3 of \cite{FiliSamuels} we obtain that
	\begin{align*}
		M_\infty(\alpha) & = M_\infty(\zeta\alpha^L) \\
			& \leq \max_{1\leq n\leq N}\{M(\zeta_n\alpha_n^{L_n})\} \\
			& \leq \max_{1\leq n\leq N}\{M(\zeta_n\alpha_n^{L_n})^{S_n}\} \\
			& = \max_{1\leq n\leq N}\{M(\alpha_n)\} \\
			& = M_\infty(\alpha).
	\end{align*}
	Therefore, we have that $M_\infty(\alpha) = \max_{1\leq n\leq N}\{M(\zeta_n\alpha_n^{L_n})\}$.  As we have noted,
	each element $\zeta_n\alpha_n^{L_n}$ belongs to $K_\alpha$ completing the proof of the first statement.
	
	Now we have that $M_\infty(\alpha) = M(\beta)$ for some $\beta\in K_\alpha$.  Since $K_\alpha$ is Galois, it must contain
	all conjugates of $\beta$ over $\rat$, and therefore, it contains the product of all roots outside the unit circle.  This
	product is a real number so $K_\alpha$ must contain its absolute value.  Hence we get that $M_\infty(\alpha)\in K_\alpha$.
\end{proof}

In the case of $M_1$, we cannot establish a result as strong as Theorem \ref{StrongLocation}, but we can prove an analog of its
second statement.

\begin{thm} \label{Location}
	If $\alpha$ is an algebraic number then $M_1(\alpha) \in K_\alpha$.
\end{thm}
\begin{proof}
	By Theorem \ref{Achieved}, we know that there exist $\alpha_1,\ldots,\alpha_N\in\rad(K_\alpha)$ such that $\alpha = \alpha_1\cdots\alpha_N$ and
	\begin{equation*}
		M_1(\alpha) = M(\alpha_1)\cdots M(\alpha_N).
	\end{equation*}
	According to Lemma \ref{HeightInK}, for each $n$ there exists an algebraic number $\gamma_n\in K_\alpha$ and a positive integer $S_n$ such that
	$M(\alpha_n) = M(\gamma_n)^{S_n}$.  Each conjugate of $\gamma$ over $\rat$ must belong to the Galois extension $K_\alpha$, which implies
	that $M(\gamma_n)\in K_\alpha$ for all $n$.  It follows that $M_1(\alpha)\in K_\alpha$.
\end{proof}

\end{document}